\documentclass [a4]{article}

\usepackage{wrapfig}
\usepackage[all]{xy}
\usepackage{amsmath,amssymb}
\usepackage{amsthm}
\usepackage{color}

\theoremstyle{plain}
  \newtheorem{theorem}{Theorem}[section]
  \newtheorem{corollary}[theorem]{Corollary}
  \newtheorem{lemma}[theorem]{Lemma}
  
  \newtheorem{proposition}[theorem]{Proposition}

\theoremstyle{definition}
  \newtheorem{definition}[theorem]{Definition}
  \newtheorem{assumption}[theorem]{Assumption}
  \newtheorem{example}[theorem]{Example}
  
  \newtheorem{remark}[theorem]{Remark}
\newcommand{\Z}{\mathbb{Z}}
\newcommand{\Q}{\mathbb{Q}}

\newcommand{\Cat}{\mathbf{Cat}}
\newcommand{\Set}{\mathbf{Set}}
\newcommand{\Ch}{\mathbf{Ch}}
\newcommand{\Top}{\mathbf{Top}}
\newcommand{\Tcat}{\mathbf{TCat}}
\newcommand{\op}{\mathrm{op}}

\newcommand{\A}{\mathcal{A}}
\newcommand{\B}{\mathcal{B}}
\newcommand{\T}{\mathcal{T}}
\newcommand{\V}{\mathcal{V}}
\newcommand{\F}{\mathcal{F}}
\newcommand{\W}{\mathcal{W}}
\newcommand{\C}{\mathcal{C}}
\newcommand{\D}{\mathcal{D}}
\newcommand{\E}{\mathcal{E}}
\newcommand{\M}{\mathcal{M}}

\newcommand{\Ho}{\mathrm{Ho}}

\author{Kazunori Noguchi and Kohei Tanaka}
\title {\textbf{The Euler characteristic of an enriched category}}

\begin{document}

\renewcommand{\thesection}{\arabic{section}}

\maketitle

\textbf{2010 Mathematics Subject Classification :} 18D20; 55U35

\textbf{Keywords :} Euler characteristic, enriched categories, monoidal model categories

%%%%%%%%%%%% section 1 %%%%%%%%%%%%%%%%%%%%

\begin{abstract}
We define the Euler characteristic of a category enriched by a monoidal model category. If a monoidal model category $\V$ is equipped with an Euler characteristic that is compatible with weak equivalences and  fibrations in $\V$, then our Euler characteristic of $\V$-enriched categories is also compatible with weak equivalences and fibrations in the canonical model structure on the category of $\V$-enriched categories induced by that of $\V$. In particular, we focus on the case of topological categories; i.e., categories enriched by the category of topological spaces. As its application, we obtain the ordinary Euler characteristic of a cellular stratified space $X$ by computing the Euler characteristic of the face category $C(X)$ induced from $X$.
\end{abstract}

\tableofcontents

\section{Introduction}

The Euler characteristic of a topological space is a classical homotopy invariant. However, the Euler characteristic is defined not only for topological spaces, but also for finite posets \cite{Rot64}, groupoids \cite{BD01}, and categories \cite{Lei08a}, \cite{BL08}, \cite{FLS11}, \cite{Nog11}, \cite{Nog13}. Moreover, Leinster defined an invariant for categories enriched by a monoidal category, called {\em magnitude} \cite{Lei13}. Our work in this paper is based on magnitude, so we give a quick review of it here.

Let $\V$ be a monoidal category and let $k$ be a ring. For finite sets $I$ and $J$, an $I\times J$ \textit{matrix} over $k$ is a function $I\times J\to k$. For an $I\times J$ matrix $\zeta$ and a $J\times H$ matrix $\xi$, the $I\times H$ matrix $\zeta \xi$ is defined by $\zeta \xi(i,h)=\sum_j \zeta(i,j)\xi(j,h)$ for any $i$ in $I$ and $h$ in $H$. An $I\times J$ matrix $\zeta$ has a $J\times I$ transpose $\zeta^*$. Given a finite set $I$, we write $u_I: I\to k$ (or simply $u$) for the column vector with $u_I(i)=1$ for all $i$ in $I$. Let $\zeta$ be an $I\times J$ matrix over $k$. A \textit{weighting} on $\zeta$ is a column vector $w:J\to k$ such that $\zeta w=u_I$. A \textit{coweighting} on $\zeta$ is a row vector $v:I\to k$ such that $v \zeta =u^*_I$. The matrix $\zeta$ has \textit{magnitude} if it admits a weighting and a coweighting. Its \textit{magnitude} is then
$$|\zeta|=\sum_jw(j)=\sum_i v(i)\in k.$$
This definition does not depend on the choice of a weighting and a coweighting.

Let $$|-|:(\V_0/\cong ,\otimes,\mathbf{1})\longrightarrow (k,\cdot,1)$$
be a monoid homomorphism where $\V_0$ is the collection of objects of $\V$ and $\V_0/\cong$ denotes the  isomorphism classes of objects of $\V$. For a $\V$-category $\A$ having finitely many objects, the \textit{similarity matrix} of $\V$ is the $\A_0\times \A_0$ matrix $\xi_{\A}$ over $k$ defined by $\xi_{\A}(a,b)=|\V(a,b)|$ for any objects $a$ and $b$ of $\A$. If $\xi_{\A}$ has a magnitude, then $\V$-category $\A$ has a \textit{magnitude} ; its magnitude is then $|\A|=|\xi_{\A}|$. In particular, Leinster studied the magnitude of finite metric spaces. See \cite{Lei13} for more details.

In this paper, we consider the case in which a monoidal category is equipped with a model structure. A \textit{model structure} on a category consists of three classes of morphisms, called \textit{weak equivalences, fibrations, and cofibrations}, satisfying certain conditions, and this is a framework to do homotopy theory \cite{Qui67}. If a monoidal model category $\V$ satisfies certain conditions, a model structure is induced on the category $\V\textbf{Cat}$ of categories enriched by $\V$, called the \textit{canonical model structure}. Suppose that 
$$|-|:(\V_0/\cong, \otimes, \mathbf{1})\longrightarrow (k,\cdot,1)$$
is a monoid homomorphism compatible with weak equivalences and fibrations in $\V$; i.e., for a weak equivalence $f:X \to Y$ in $\V$, we have $|X|=|Y|$ and for a fibration $p: E \to B$ in $\V$ with fiber $F$, we have $|E|=|B| \cdot |F|$. This is a natural assumption for topologists when we regard $|-|$ as the standard topological Euler characteristic. By following Leinster's work, we can define an invariant of $\V$-categories induced by $|-|$. Then, one can ask whether the induced invariant is compatible with weak equivalences and fibrations in $\V\textbf{Cat}$. The following is a positive answer to the question:

\begin{theorem}\label{main} Suppose that the category of $\V$-enriched categories admits the canonical model structure and
$\chi(-)$ is the Euler characteristic of $\V$-enriched categories.
\begin{enumerate}
\item If a $\V$-functor $f:\A\to \B$ is a weak equivalence, then the Euler characteristics of $\A$ and $\B$ are equal; i.e., $\chi(\A)=\chi(\B)$.
\item If a $\V$-functor $p : \E \to \B$ is a fibration with fiber $\F$ satisfying certain conditions, then 
we have $\chi(\E)=\chi(\B)\chi(\F)$.
\end{enumerate}
\end{theorem}

Computing the Euler characteristic of cellular stratified spaces is an application of the Euler characteristic of $\V$-categories. 
A \textit{cellular stratified} space is a space of attached cells with possibly incomplete boundaries, inductively \cite{Tama}.
If $X$ is a finite cell complex, the Euler characteristic $\chi(X)$ is obtained by counting the numbers of $n$-cells, 
however in the case of cellular stratified spaces this does not work. For example, 
the half-open interval $(0,1]$ is a cellular stratified space consisting of a $1$-cell $(0,1]$ and a $0$-cell $\{1\}$.
We have $\chi\left( (0,1]\right)=1$ since $(0,1]$ is contractible, however the alternating sum of the numbers of $n$-cells is $1-1=0$.

Tamaki shows that a nice cellular stratified space $X$ is homotopy equivalent to the classifying space $BC(X)$ of the face category $C(X)$ enriched by the category of topological spaces in \cite{Tama}. From the following theorem, we can show that the standard Euler characteristic $\chi(X)$ is equal to our Euler characteristic $\chi(C(X))$ of the topological category $C(X)$.

\begin{theorem}\label{main2}
Let $\A$ be a certain topological category. Then, the Euler characteristic $\chi(\A)$ coincides with the topological Euler characteristic $\chi(B\A)$ 
of the classifying space $B\A$ of $\A$.
\end{theorem}

\begin{corollary}\label{main3}
Let $X$ be a certain cellular stratified space. Then, we have $\chi(X)=\chi(C(X))$ where $C(X)$ is the cylindrical face category of $X$.
\end{corollary}

This paper is organized as follows. 
In Section 2, we give a review of enriched categories and model categories including the canonical model structure on the category of enriched categories.
In Section 3, we introduce the Euler characteristic of $\V$-enriched categories for a monoidal model category $\V$.
We give a proof of Theorem \ref{main}.
In Section 4, we focus on the case in which $\V$ is the category of topological spaces, and we state Theorem \ref{main2} and Corollary \ref{main3}.

\section{Review of enriched categories and model categories}

We first begin with a brief review of basic notions of enriched categories.
Let $(\V,\otimes,\mathbf{1})$ denote a monoidal category throughout this paper.
A $\V$-enriched category is a generalized idea of a category using hom-objects in $\V$
instead of hom-sets. See more details in \cite{Kel05}. 

\begin{definition}[Enriched category]
A $\V$-enriched category, or simply, a $\V$-category $\A$ consists of a set of objects $\A_{0}$ and 
a {\em hom-object} $\A(a,b)$ of $\V$ for each pair of objects $a$ and $b$ of $\A$ 
with a composition $\circ : \A(b,c) \otimes \A(a,b) \to \A(a,c)$ satisfying the coherence condition. We call $\A$ {\em finite} if $\A_{0}$ is finite.

A {\em $\V$-functor} $f : \A \to \B$ between two $\V$-categories $\A$ and $\B$ 
consists of a map $f : \A_{0} \to \B_{0}$ on objects and a morphism 
$\A(a,b) \to \B(fa,fb)$ in $\V$ for each pair of objects $a,b \in \A_{0}$ preserving the composition and identities.
\end{definition}

The following model structure on the category of small categories is called the folk model structure (see \cite{JT91}).
It is closely related to the model structure on the category of $\V$-categories.

\begin{definition}
A functor $p : E \to B$ between small categories is called an {\em isofibration} 
if for any isomorphism $f : p(e) \to b$ in $B$, there exists an isomorphism $g : e \to e'$ in $E$ such that $p(g)=f$.
On the other hand, a functor $i : X \to Y$ is called an {\em isocofibration} if it is injective on objects.
\end{definition}

\begin{theorem}[Folk model structure]
The category of small categories admits the following model structure called the folk model structure.
\begin{itemize}
\item The weak equivalences are equivalences of categories.
\item The fibrations are isofibrations.
\item The cofibrations are isocofibrations. 
\end{itemize}
\end{theorem}

When $\V$ is a nice monoidal model category \cite{Hov99},
the category $\V\Cat$ of small $\V$-categories also admits a model structure called canonical. 
This is a mixture of the model structure on $\V$ and the folk model structure on the category of small categories. 
In order to define the canonical model structure, we need the category of connected components 
$\pi_{0}\A$ of a $\V$-category $\A$.

\begin{definition}[Homotopy category]
Let $\V$ be a monoidal model category.
The homotopy category $\Ho(\V)$ is the category consisting of 
$\Ho(\V)_{0}=\V_{0}$ and $\Ho(\V)(X,Y)=\V(QRX,QRY)/\simeq$, 
where $QR$ is the cofibrant and fibrant replacement and $\simeq$ is the homotopy relation.
Note that the monoidal structure on $\V$ induces that of $\Ho(\V)$.
\end{definition}

Generally, the homotopy category $\Ho(\V)$ is defined as the localization $\gamma : \V \to \V[W^{-1}]=\Ho(\V)$ 
with respect to the class of weak equivalences $W$. 
It is determined uniquely up to equivalence of categories.

\begin{remark}[Theorem 1.2.10 of \cite{Hov99}]
If $X$ is cofibrant and $Y$ is fibrant in a model category $\V$,
then there exists a natural isomorphism $\Ho(\V)(X,Y) \cong \V(X,Y)/\simeq$.
\end{remark}

\begin{definition}[Connected component]
For an object $X$ in a monoidal model category $\V$, the \textit{connected components} $\pi_{0}X$ of $X$
is the set of morphisms $\Ho(\V)(\mathbf{1},X)$ in the homotopy category. Moreover, for a $\V$-category $\A$,
the \textit{category of connected components} $\pi_{0}\A$ of $\A$ is a small category whose set of objects is $\A_{0}$ and 
set of morphisms from an object $a$ to an object $b$ of $\A$ is $(\pi_{0}\A)(a,b)=\pi_{0}(\A(a,b))$.
\end{definition}

\begin{definition}
Let $\V$ be a monoidal model category.
\begin{enumerate}
\item We call a $\V$-functor $f : \A \to \B$ to be a {\em DK-equivalence} if $f : \A(x,y) \to \B(fx,fy)$
is a weak equivalence in $\V$ for any objects $x$ and $y$ of $\A$ and 
$\pi_{0}f : \pi_{0}\A \to \pi_{0}\B$
is an equivalence of categories.
\item We call a $\V$-functor $p : \A \to \B$ to be a {\em naive fibration} if $p : \A(x,y) \to \B(px,py)$
is a fibration in $\V$ for any objects $x$ and $y$ of $\A$ and 
$\pi_{0}p : \pi_{0}\A \to \pi_{0}\B$
is an isofibration of categories.
\end{enumerate}
\end{definition}

\begin{definition}[The canonical model structure]\label{canonical_model}
Let $\V$ be a monoidal model category. A model structure on the category of small $\V$-categories is called {\em canonical} 
if it has the following properties:
\begin{enumerate}
\item A $\V$-category $\A$ is fibrant if and only if $\A(a,b)$ is fibrant in $\V$ for any objects $a$ and $b$ of $\A$. 
\item A $\V$-functor $f : \A \to \B$ is a trivial fibration if and only if $f : \A(a,b) \to \B(fa,fb)$ is a fibration in $\V$ for any objects $a$ and $b$ of $\A$ and surjective on objects.
\item The weak equivalences are DK-equivalences. 
\item The fibrations between fibrant objects are naive fibrations.
\end{enumerate}
\end{definition}

The definition of the canonical model structure in \cite{BM13} only requires conditions 1 and 2 stated above.
However, to clarify the relation with DK-equivalences and naive fibrations, we use the above definition. 

A natural question is when $\V\Cat$ admits the canonical model structure.

\begin{example} Here are some examples for which the canonical model structure is known to exist.
\begin{itemize}
\item A standard category is a category enriched by the category of sets $(\Set, \times , *)$ equipped with the trivial model structure. The category of small categories $\Cat$ admits the canonical model structure that coincides with the folk model structure \cite{JT91}.
\item A topological category is a category enriched by the category of 
compactly generated weak Hausdorff spaces $(\Top, \times, *)$ equipped with the classical model structure.
The category of small topological categories $\Tcat$ admits the canonical model structure \cite{Amr}.
\item A 2-category is a category enriched by the category of small categories $(\Cat, \times, *)$ with the folk model structure \cite{JT91}. The category of 2-categories admits the canonical model structure \cite{Lac02}.
\item A simplicial category is a category enriched by the category of simplicial sets $(\Set^{\Delta^{\op}}, \times, *)$ equipped with the classical model structure. The category of simplicial categories admits the canonical model structure \cite{Ber07}.
\item A DG category over a ring $R$ is a category enriched by the category of chain complexes $(\Ch_{R}, \otimes, R)$ with the projective model structure. The category of DG-categories admits the canonical model structure \cite{Tab05}.
\end{itemize}
\end{example}

Berger and Moerdijk found a general condition on $\V$ for existence of the canonical model structure in Theorem 1.9 of \cite{BM13}.

\section{The Euler characteristic of enriched categories}

In this section, we focus on homotopical properties of the Euler characteristic.

\subsection{The Euler characteristic of $\V$-categories}

The Euler characteristic(magnitude) of $\V$-categories in \cite{Lei13} is constructed by a monoid homomorphism 
$$|-| : (\V_{0}/\cong, \otimes, \mathbf{1}) \to (k,\cdot,1)$$
from the set of isomorphism classes of objects to a ring $k$.
When $\V$ is a monoidal model category, 
we require the homomorphism to be an invariant with respect to weak equivalences.
In other words, we define the homomorphism as one from the set of isomorphism classes of objects in the homotopy category 
$\Ho(\V)$.

\begin{definition}
Let $\V$ be a monoidal model category, and
let $\W$ be a full subcategory of $\Ho(\V)$ closed under finite coproducts and direct summands. 
A {\em measure} of $\V$ is a homomorphism (preserving tensor products and coproducts)
$$|-| : (\W_{0}/\cong, \otimes, \mathbf{1}) \longrightarrow (k, \cdot , 1)$$
from the set of isomorphism classes of $\W$ to a ring $k$. We write it simply as $|-| : \W \to k$.
\end{definition}

\begin{definition}
For a monoidal model category $\V$ with a measure $|-| : \W \to k$, 
we call a $\V$-category $\A$ to be \textit{measurable on $\mathcal{W}$ by $|-|$} if $\A(a,b)$ belongs to $\W_{0}$ for any objects $a$ and $b$ of $\A$.
\end{definition}

\begin{definition}
Let $\V$ be a monoidal model category with a measure, 
and let $\A$ be a finite measurable $\V$-category.
\begin{enumerate}
\item The {\em similarity matrix} of $\A$ is the function 
$\xi : \A_{0} \times \A_{0} \to k$ given by $\xi(a,b)=|\A(a,b)|$.
\item Let $u$ denote the column vector with $u(a)=1$ for any object $a$ of $\A$.
A weighting on $\A$ is a column vector $w:\A_0\to k$ such that $\xi w=u$. 
Dually, a coweighting on $\A$ is a row vector $v:\A_0\to k$ such that $v \xi = u^{*}$.
\end{enumerate}
\end{definition}

Note that we have
$$\sum_{i \in \A_{0}}w_{i} = u^{*} w = v \xi w= v u=\sum_{j \in \A_{0}} v_{j}$$
if both a weighting and a coweighting exist. 
Moreover, 
$$\sum_{i} w_{i}= u^{*} w = w^{*} u = w^{*} \xi w' = u^{*} w' = \sum_{i} w'_{i}$$
for two (co)weightings $w$ and $w'$ on $\V$, and the equality guarantees the following definition.

\begin{definition}
Let $\V$ be a monoidal model category with a measure, and let $\A$ be a finite measurable $\V$-category. 
We say that $\A$ \textit{admits Euler characteristic} if it has 
both a weighting $w$ and a coweighting $v$ on $\A$. 
Then the \textit{Euler characteristic} of $\A$ is defined by
$$\chi(\A) = \sum_{i}w(i) = \sum_{j}v(j).$$
\end{definition}

\begin{remark}\label{inversion}
Let $\V$ be a monoidal model category with a measure, and
let the similarity matrix $\xi$ of a finite measurable $\V$-category $\A$ 
have an inverse matrix $\xi^{-1} : \A_{0} \times \A_{0} \to k$. Then there exist 
a weighting $w_{b}=\sum_{a \in \A_{0}} \xi^{-1}(a,b)$ and a coweighting $v_{a}=\sum_{b \in \A_{0}}\xi^{-1}(a,b)$.
Hence, $\A$ admits Euler characteristic, and we have $\chi(\A)= \sum_{a,b \in \A_{0}} \xi^{-1}(a,b)$.
\end{remark}

\begin{lemma}\label{coprod}
For an initial object $\phi$ of a monoidal model category $\V$ with a measure
$|-| : \W \to k$, we have $|\phi|=0$.
\end{lemma}
\begin{proof}
The universality of coproducts shows that $X \coprod \phi \cong X$.
Since the measure preserves finite coproducts, we have
$|X|+|\phi|=|X \coprod \phi|=|X|$. Note that every isomorphism in a model category is a weak equivalence (see the middle of p. 6 of \cite{Hov99}). Thus, $|\phi|=0$.
\end{proof}

The category of $\V$-categories also admits a monoidal structure if 
$\V$ is symmetric.
The tensor product $\C \otimes \D$ of two $\V$-categories $\C$ and $\D$ is given by $(\C \otimes \D)_{0}=\C_{0} \times \D_{0}$ and 
$$(\C \otimes \D)((c_{1},d_{1}),(c_{2},d_{2})) = \C(c_{1}, c_{2}) \otimes \D(d_{1},d_{2})$$
for objects $c_1$ and $c_2$ of $\C$ and objects $d_1$ and $d_2$ of $\D$.

\begin{proposition}
Let $\V$ be a monoidal model category with a measure.
Suppose that $\A_{1},\A_{2},\dots,\A_{n}$ are $\V$-categories admitting Euler characteristics. Then, 
\begin{enumerate}
\item $\chi( \otimes_{i} \A_{i})= \prod_{i} \A_{i}$ when $\V$ is a symmetric monoidal model category. \\
\item $\chi(\coprod_i \A_{i})= \sum_i \chi(\A_{i})$.
\end{enumerate}
\end{proposition}
\begin{proof}
This is shown similarly to Proposition 2.6 in \cite{Lei08a},
since the measure preserves tensor products and coproducts.
\end{proof}

\subsection{Weak equivalences and the Euler characteristic}

We examine relations between the Euler characteristics of $\V$-categories and the canonical model structure 
on the category of $\V$-categories.

\begin{definition}
Let $\V$ be a monoidal model category.
Two objects $a$ and $b$ in a $\V$-category $\A$ 
are \textit{weakly equivalent} if they are isomorphic in the category of connected components $\pi_{0}\A$ of $\A$. 
\end{definition}

\begin{lemma}\label{weakly_equivalent}
Let $\V$ be a monoidal model category.
For two weakly equivalent objects $a$ and $b$ of a $\V$-category $\A$, the two hom-objects 
$\A(a,c)$ and $\A(b,c)$ (resp. $\A(c,a)$ and $\A(c,b)$) are weakly equivalent to each other in $\V$ for any object $c$  of $\A$.
\end{lemma}
\begin{proof}
In the category of connected components $\pi_{0}\A$, there exists an isomorphism $[\alpha]$ of $\pi_{0}\A(a,b)$.
This is the homotopy class of a morphism $\alpha : \mathbf{1} \to \A(a,b)$ in $\Ho(\V)$.
The following composition 
$$\A(b,c) \cong \A(b,c) \otimes \mathbf{1} \stackrel{1 \otimes \alpha}{\longrightarrow} 
\A(b,c) \otimes \A(a,b) \stackrel{\circ}{\longrightarrow} \A(a,c)$$
is an isomorphism in $\Ho(\V)$.
\end{proof}

The following proposition can be shown similarly to Lemma 1.12 in \cite{Lei08a}.

\begin{proposition}\label{admitting}
Let $\V$ be a monoidal model category with a measure.
Suppose that two finite measurable $\V$-categories $\A$ and $\B$ are DK-equivalent.
Then, $\A$ admits a (co)weighting if and only if $\B$ does.
\end{proposition}
\begin{proof}
Suppose that $\B$ has a weighting $w_{\B}$. 
There exists an equivalence of categories $f : \pi_{0}(\A) \to \pi_{0}(\B)$ between the categories of connected components.
By Lemma \ref{weakly_equivalent}, the similarity matrix $\xi_{\A}$ on $\A$ satisfies 
$\xi_{\A}(a,c)=\xi_{\A}(b,c)$ and $\xi_{\A}(c,a)=\xi_{\A}(c,b)$ for weakly equivalent objects $a,b$ and any object $c$ in $\A$.
Let $C(a)$ denote the number of objects of the weak equivalence class of $a$ of $\A$, similarly $C(b)$ for $b \in \B_{0}$.
A weighting $w_{\A}$ on $\A$ is given by
$$w_{\A}(a)= \left(C(fa)/C(a) \right) w_{\B}(fa).$$
\end{proof}

\begin{theorem}
Let $\V$ be a monoidal model category with a measure.
Suppose that two finite measurable $\V$-categories $\A$ and $\B$ are DK-equivalent.
Then, their Euler characteristics are equal $\chi(\A)=\chi(\B)$ if $\A$ or $\B$ admits Euler characteristic.
\end{theorem}
\begin{proof}
There exists an equivalence of categories $f: \pi_{0}\A \to \pi_{0}\B$ and 
the inverse $g : \pi_{0}\B \to \pi_{0}\A$. For any pair of objects $(a,b)$  of $\A \times \B$, we have
$$\xi_{\A}(a,gb)= \xi_{\B}(fa,fgb)=\xi_{\B}(fa,b).$$
By Proposition \ref{admitting}, both $\A$ and $\B$ admit Euler characteristic.
Let $w_{\A}$ be a weighting on $\A$ and $v_{\B}$ be a coweighting on $\B$. Then we have
\begin{equation}
\begin{split}
\chi(\A) &=\sum_{a \in \A_{0}} w_{\A}(a) \\ \notag
&= \sum_{a \in \A_{0}} \sum_{b \in \B_{0}}v_{\B}(b) \xi_{\B}(fa,b) w_{\A}(a) \\
&=\sum_{a \in \A_{0}} \sum_{b \in \B_{0}}v_{\B}(b) \xi_{\A}(a,gb) w_{\A}(a) \\
&=\sum_{b \in \B_{0}}v_{\B}(b) \\
&=\chi(\B). \\
\end{split}
\end{equation}
\end{proof}

\subsection{Fibrations and the Euler characteristic}

Another important homotopical property of the Euler characteristic 
is the product formula with respect to fibrations. 
In the category of topological spaces, 
a certain fibration $p : E \to B$ over a connected base $B$ with fiber $F$ induces the equation $\chi(E)=\chi(B) \chi(F)$.
When $B$ has connected components $B_{i}$, the above equation is generally extended as the following form
$$\chi(E) = \sum_{i \in \pi_{0}B} \chi(B_{i})\chi(F_{i}),$$
where $F_{i}$ is the fiber over a point of $B_{i}$.
We focus on the relation between the Euler characteristic of $\V$-categories 
and naive fibrations.

\begin{definition}
Let $X$ be an object in a monoidal model category $\V$.
We call $X$ {\em connected} if the set of connected components $\pi_{0}X=\Ho(\V)(\mathbf{1},X)$ is a single point.
On the other hand, a small category $\C$ is called {\em connected} when there exists a zigzag sequence of morphisms 
$$x \longrightarrow x_{1} \longleftarrow x_{2} \longrightarrow \cdots \longrightarrow x_{n-1} \longleftarrow y$$
starting at $x$ and ending at $y$ for any two objects $x$ and $y$ of $\C$.
We call a $\V$-category $\A$ {\em connected} if the category $\pi_{0}\A$ of connected components is connected.
Moreover, we call $\A$ {\em strongly connected} if $\A$ is connected and every hom-object $\A(x,y)$ is connected in $\V$.
\end{definition}

\begin{assumption}\label{assumption_V}
Assume that our monoidal model category $\V$ satisfies the following conditions:
\begin{enumerate}
\item $\V$ is cartesian with the cofibrant unit.
\item The set of connected components $\pi_{0}(X)$ is empty if and only if $X$ is the initial object $\phi$ in $\V$.
\item The category of $\V$-categories admits the canonical model structure (Definition \ref{canonical_model}).
\end{enumerate}
Regarding the first condition, a cartesian monoidal category is a monoidal category whose tensor product $\otimes$ coincides with the categorical product.
Hence, unit object $\mathbf{1}$ is the terminal object. In particular, $\V$ is symmetric.
\end{assumption}

\begin{lemma}
Let $\V$ be a monoidal model category satisfying Assumption \ref{assumption_V}. 
The unit $\mathbf{1}$ of $\V$ is connected.
\end{lemma}
\begin{proof}
Since $\mathbf{1}$ is the terminal object, this is fibrant and
the set of morphisms $\V(\mathbf{1},\mathbf{1})$ is a single point.
Then we have
$$\pi_{0}(\mathbf{1}) = \Ho(\V)(\mathbf{1},\mathbf{1}) \cong \V(\mathbf{1}, \mathbf{1})/\simeq = *.$$
\end{proof}

\begin{lemma}
Let $\V$ be a monoidal model category satisfying Assumption \ref{assumption_V}.
Any $\V$-category can be decomposed as $\A=\coprod_{i} \A_{i}$ for connected $\V$-categories $\A_{i}$.
\end{lemma}
\begin{proof}
The category of connected components $\pi_{0}\A$ can be decomposed as 
$\pi_{0}\A = \coprod_{i} \B_{i}$ for connected subcategories $\B_{i}$ of $\pi_{0}\A$.
Consider the full subcategory $\A_{i}$ of $\A$ having the same objects as $\B_{i}$.
Then $\A_{i}$ is connected and $\A= \coprod_{i} \A_{i}$ since 
$\A(a,b)=\phi$ if and only if $\pi_{0}\A(a,b)=\phi$ by Assumption \ref{assumption_V}.
\end{proof}

\begin{definition}
Let $\V$ be a monoidal category. 
Let $p : E \to B$ be a morphism in a monoidal category $\V$. The fiber $F_{b}$ 
of $p$ over $b : \mathbf{1} \to B$ is defined by the pullback 
\[
\xymatrix{
F_{b} \ar[r] \ar[d] & E \ar[d]^{p} \\
\mathbf{1} \ar[r]_{b} & B
}
\]
in $\V$. The fiber $\F_{b}$ of a $\V$-functor $p : \E \to \B$ over an object $b$ of $\B$ can be defined in $\V\Cat$ similarly to 
the above case. Note that choosing an object $b$ of $\B$ gives a $\V$-functor 
$\mathcal{I} \to \B$ from the unit category $\mathcal{I}$ consisting of a single object and the unit $\mathbf{1}$ as the hom-object.
The fiber $\F_{b}$ is a $\V$-category consisting of the inverse image $p^{-1}(b)$ as objects and the fiber 
of $p : \E(x,y) \to \B(b,b)$ over the identity $1_{b} : \mathbf{1} \to \B(b,b)$ as the hom-object $\F_{b}(x,y)$.
\end{definition}

\begin{lemma}\label{fiber}
Let $\V$ be a monoidal model category satisfying Assumption \ref{assumption_V}. 
Suppose that $p : E \to B$ is a fibration between fibrant objects in $\V$ with the connected base $B$.
Then, any two fibers $F_{b}$ and $F_{b'}$ are weakly equivalent.
\end{lemma}
\begin{proof}
The morphisms $b, b' : \mathbf{1} \to B$ are homotopic to each other. Hence, the fibers $F_{b}$ and $F_{b'}$ are weakly equivalent by Corollary \ref{stable3}. 
\end{proof}

Unfortunately, the fibers of a naive fibration are not DK-equivalent to each other in general even if the base $\V$-category is connected.
For example, a finite left $G$-set $A$ for a finite nontrivial group $G$ gives a category $A_{G}$ with two obejcts $0,1$ and 
$A_{G}(0,0)=G$, $A_{G}(0,1)=A$, $A_{G}(1,0)=\phi$, and $A_{G}(1,1)=*$.
The constant function induces a isofibration $p$ from $A_{G}$ to the category formed by $0 \to 1$. 
However, these fibers $F_{0}=G$ and $F_{1}=*$.
Thus, we require the category of connected components to be a groupoid.

\begin{lemma}\label{fiber2}
Let $\V$ be a monoidal model category satisfying Assumption \ref{assumption_V}. 
Suppose that $p : \E \to \B$ is a naive fibration between fibrant $\V$-categories and let $\pi_{0}\B$ be a connected groupoid.
Then, any two fibers $\F_{b}$ and $\F_{b'}$ are DK-equivalent.
\end{lemma}
\begin{proof}
Since $\pi_{0}\B$ is a connected groupoid, two objects $b,b'$ are isomorphic in $\pi_{0}B$.
This gives two morphisms $\alpha : \mathbf{1} \to \B(b,b')$ and $\beta : \mathbf{1} \to \B(b',b)$ satisfying the identity condition.
Let $\mathcal{H}$ be a $\V$-category consisting of two objects $0$, $1$ and $\mathcal{H}(i,j)=\mathbf{1}$ for any $0 \leqq i,j \leqq 1$.
This is DK-equivalent to the unit $\V$-category $\mathcal{I}$ consisting of only 
a single point as the object and the unit $\mathbf{1}$ as the hom-object. 
The two morphisms $\alpha$, $\beta$ give rise to a $\V$-functor $h : \mathcal{H} \to \B$ making the following diagram commute in $\Ho(\V)$
\[
\xymatrix{
\mathcal{I} \coprod \mathcal{I} \ar[r]^{b \coprod b'} \ar[d]_{} & \mathcal{B} \ar@{=}[d] \\
\mathcal{H} \ar[r]_{h} & \mathcal{B}.
}
\]
We have the following homotopy commutative diagrams 
\[
\xymatrix{
\mathcal{I} \ar[r]^{b} \ar[d] & \B \ar@{=}[d] & \E \ar[l]_{p} \ar@{=}[d] \\
\mathcal{H} \ar[r]^{h}  & \B & \E \ar[l]_{p}  \\
\mathcal{I} \ar[r]^{b'} \ar[u] & \B \ar@{=}[u] & \E. \ar[l]_{p} \ar@{=}[u]
}
\]
By Corollary \ref{stable3}, we conclude that $\F_{b}$ and $\F_{b'}$ are DK-equivalent.
\end{proof}

\begin{lemma}\label{groupoid}
Let $\V$ be a monoidal model category with a measure satisfying Assumption \ref{assumption_V}.
Suppose that $\B$ is a $\V$-category admitting Euler characteristic and $\pi_{0}\B$ is a connected groupoid.
Then, we have $\chi(\B) = |\B(b,b)|^{-1}$ for any object $b$ of $\B$.
\end{lemma}
\begin{proof}
Fix an object $b$ of $\B$. 
Lemma \ref{weakly_equivalent} shows that $|\B(x,y)| = |\B(b,b)|$ for any objects $x$ and $y$ of $\B$ since $\pi_{0}\B$ is a connected groupoid.
The similarity matrix $\xi_{\B}$ is $\xi_{\B}(x,y)=|\B(b,b)|$ for any objects $x$ and $y$ of $\B$.
A weighting $w$ on $\B$ can be defined as $w(x)= (\B_{0}^{\sharp} \cdot |\B(b,b)|)^{-1}$, where $\B_{0}^{\sharp}$ is the cardinal of $\B_{0}$.
We obtain the Euler characteristic $\chi(\B) = \sum_{x \in \B_{0}} w(x) = |\B(b,b)|^{-1}$.
\end{proof}

\begin{lemma}
Let $\V$ be a monoidal model category satisfying Assumption \ref{assumption_V}. 
If $p : E \to B$ is a fibration between fibrant objects, then $\pi_{0}F$ is isomorphic to the fiber of $\pi_{0}p : \pi_{0}E \to \pi_{0}B$.
\end{lemma}
\begin{proof}
A point $[e]$ in the fiber of $\pi_{0}p$ over $[b]$ of $\pi_{0}B$ induces the following homotopy commutative diagram
\[
\xymatrix{
\mathbf{1} \ar[r]^{e} \ar@{=}[d] & E \ar[d]^{p} \\
\mathbf{1} \ar[r]_{b} & B.
}
\]
There exists a unique morphism 
from $\mathbf{1}$ to the homotopy pullback $P=E \times_{B}^{h} \mathbf{1}$ (see Definition \ref{hopu} and Theorem \ref{unique_mor}) making the diagram commute up to homotopy.
This implies that 
$\pi_{0}(P)$ is isomorphic to the fiber of $\pi_{0}p : \pi_{0}E \to \pi_{0}B$.
The three objects $E$, $B$, and $\mathbf{1}$ are fibrant and $p$ is a fibration. 
The homotopy fiber $P$ and the fiber $F$ of $p$ are weakly equivalent by Theorem \ref{back}.
Hence $\pi_{0}(F) \cong \pi_{0}(P)$.
\end{proof}

\begin{corollary}\label{pi_0}
Let $\V$ be a monoidal model category satisfying Assumption \ref{assumption_V}.
Let $p : \E \to \B$ be a naive fibration between fibrant $\V$-categories. Then, the category of connected components $\pi_{0}(\F_{b})$ of fiber of $p$ over an object $b$ of $\B$ 
is isomorphic to the category of fiber of $\pi_{0}p : \pi_{0}\E \to \pi_{0}\B$.
\end{corollary}

\begin{definition}
Let $\V$ be a monoidal model category.
We say that a measure $|-| : \W \to k$ of $\V$ {\em preserves fibrations} 
if for any fibration $p : E \to B$ such that $E$, $B$, and any fiber belong to $\W$ and the base $B$ is decomposed as a finite coproduct $\coprod B_{i}$ of 
connected objects $B_{i}$, we have
$$|E| = \sum_{i} |B_{i}| \cdot |F_{i}|,$$
where $F_{i}$ is the fiber of $p$ over a morphism $\mathbf{1} \to B_{i}$.
\end{definition}

\begin{theorem}\label{product}
Suppose that $\V$ is a monoidal model category with a measure preserving fibrations and it satisfies Assumption \ref{assumption_V},
and let $p : \E \to \B$ be a naive fibration between fibrant $\V$-categories admitting Euler characteristic, and also any fiber admits Euler characteristic. 
If both $\pi_{0}\E$ and $\pi_{0}\B$ are groupoids, and $\B$ is strongly connected, we have 
$$\chi(\E) = \chi(\B)\chi(\F),$$
where $\F$ is the fiber of $p$ over an object of $\B$.
\end{theorem}
\begin{proof}
The $\V$-category $\E$ can be decomposed as the finite coproduct
$\E=\coprod_{i} \E_{i}$ for connected subcategories $\E_{i}$.
Choose an object $x_{i}$ of $\E_{i}$ for each $i$.
Consider the following two pullback diagrams
\[
\xymatrix{
\F_{i} \ar[r] \ar[d] & \F_{b} \ar[r] \ar[d] & \mathcal{I} \ar[d]^{b} \\ 
\E_{i} \ar[r] & \E \ar[r]_{p} & \B,
}
\]
where the $\V$-category $\F_{i}$ is the full subcategory of $\F_{b}$ having $p^{-1}(b) \cap (\E_{i})_{0}$ as objects.
We have the coproduct decomposition of $\F_{b}$ by 
$\F_{b}= \coprod_{i} \F_{i}$ since $\F(x,y)=\phi$ if $\E(x,y)=\phi$.
Corollary \ref{pi_0} shows that $\pi_{0}\F_{i}$ is a connected groupoid.
The Euler characteristic of $\F_{b}$ is 
$$\chi(\F_{b}) = \sum_{i} \chi(\F_{i}) = \sum_{i}(|\F_{i}(x_{i},x_{i})|)^{-1}.$$
A naive fibration implies that
$$p : \E(x_{i},x_{i}) \longrightarrow \B(px_{i},px_{i})$$
is a fibration in $\V$, and we have
$$|\E(x_{i},x_{i})| = |\B(px_{i},px_{i})| \cdot |\F_{i}(x_{i},x_{i})|$$
since $\B(px_{i},px_{i})$ is connected. The following calculation shows the result.
\begin{equation}
\begin{split}
\chi(\E) &=\sum_{i} \chi(\E_{i}) \\ \notag
&= \sum_{i} (|\E(x_{i},x_{i})|)^{-1}\\
&= \sum_{i} \left(|\B(px_{i},px_{i})| \cdot |\F_{i}(x_{i},x_{i})| \right)^{-1} \\
&= |\B(b,b)|^{-1} \sum_{i} \left(|\F_{i}(x_{i},x_{i})| \right)^{-1} \\
&= \chi(\B) \chi(\F).
\end{split}
\end{equation}
\end{proof}

Let us consider the case in which $\B$ is not strongly connected.
Since there is no guarantee that an object $X$ in $\V$ can be decomposed by connected objects, 
we need the following assumption.

\begin{assumption}\label{assumption}
We assume that our monoidal model category $\V$ with a measure $|-| : \W \to k$ satisfies Assumption \ref{assumption_V} and the following properties.
\begin{enumerate}
\item The tensor product $\otimes$ is compatible with the coproducts; i.e., 
$\left(\coprod_{i} X \right) \otimes \left( \coprod_{j}Y_{j} \right)$ is naturally isomorphic to 
$\coprod_{i,j} \left( X_{i} \otimes Y_{j}\right)$.
\item Any object $X$ in $\W$ has the finite connected components. 
\item Any object $X$ in $\W$ is decomposed as a finite coproduct $X=\coprod_{i \in \pi_{0}X}X_{i}$ for connected objects $X_{i}$, and the following is a pullback diagram 
\[
\xymatrix{
X_{i} \ar[r] \ar[d] & X=\coprod_{i} X_{i} \ar[d] \\
\mathbf{1} \ar[r]_{i} & \coprod_{i} \mathbf{1}.
}
\]
\end{enumerate}
\end{assumption}

Thanks to the above pullback diagram, these connected objects $X_{i}$ are determined uniquely up to isomorphism for an object $X$.
A morphism $\mu : X \otimes Y \to Z$ in $\V$ induces a map on connected components
$$\pi_{0}X \times \pi_{0}Y \longrightarrow \pi_{0}Z,\ \ (i,j) \mapsto ij.$$
When the three objects $X,Y$, and $Z$ are decomposed as 
$X= \coprod_{i \in \pi_{0}X} X_{i}$, $Y=\coprod_{j \in \pi_{0}Y}Y_{j}$, and $Z=\coprod_{k \in \pi_{0}Z} Z_{k}$ respectively, 
the pullback diagram 
\[
\xymatrix{
& Z_{ij} \ar[rr] \ar@{->}'[d][dd] & &  Z \ar[dd] \\
X_{i} \otimes Y_{j} \ar[rr] \ar[dd] \ar@{..>}[ur]^{\mu_{ij}}& & X \otimes Y \cong \coprod_{i,j} \left(X_{i} \otimes Y_{j}\right) \ar[dd] \ar[ur]^{\mu} & \\
& \mathbf{1} \ar@{->}'[r][rr]_{\hspace{-2cm}ij} & & \coprod_{k} \mathbf{1} \\
\mathbf{1} \ar[rr]_{(i,j)} \ar[ur] & & \coprod_{i,j} \mathbf{1} \ar[ur]& 
}
\]
induces a morphism $\mu_{ij} : X_{i} \otimes Y_{j} \to Z_{ij}$
making the diagram above commute.

\begin{lemma}\label{component}
Let $\V$ be a monoidal model category with a measure satisfying Assumption \ref{assumption}.
A monoid object $X$ in $\W$ induces a monoid structure on $\pi_{0}X$.
Let $X$ be decomposed as the coproduct of the connected components
$$X= \coprod_{i \in \pi_{0}X} X_{i}$$
from Assumption \ref{assumption}. If $\pi_{0}X$ is a group, all connected objects $X_{i}$ are weakly equivalent to each other.
\end{lemma}
\begin{proof}
Denote the unit of $\pi_{0}X$ by $e$.
For an element $i$ of $\pi_{0}X$, the connected components $\pi_{0}(X_{i^{-1}})=\Ho(\V)(1,X_{i^{-1}})$ consists of a single point.
The following composition
$$X_{i} \cong X_{i} \otimes \mathbf{1} \stackrel{1 \otimes i^{-1}}{\longrightarrow} X_{i} \otimes X_{i^{-1}} \stackrel{\mu_{ii^{-1}}}{\longrightarrow} X_{e}$$
is an isomorphism in $\Ho(\V)$. Thus, any connected objects $X_{i}$ is weakly equivalent to $X_{e}$.
\end{proof}

\begin{theorem}\label{product2}
Suppose that $\V$ is a monoidal model category with a measure preserving fibrations and it satisfies Assumption \ref{assumption}.
Let $p : \E \to \B$ be a naive fibration between fibrant $\V$-categories admitting Euler characteristic, and also any fiber admits Euler characteristic. 
If both $\pi_{0}\E$ and $\pi_{0}\B$ are groupoids and $\B$ is connected, we have 
$$\chi(\E) = \chi(\B)\chi(\F)$$
where $\F$ is the fiber of $p$ over an object of $\B$.
\end{theorem}
\begin{proof}
This is shown similarly to Theorem \ref{product}.
We have the finite coproduct decomposition $\E=\coprod_{i} \E_{i}$ for connected $\V$-categories $\E_{i}$ and 
$\B(b,b)=\coprod_{j} \B(b,b)_{j}$ for connected objects $\B(b,b)_{j}$.
The assumption with respect to preserving fibrations of the measure induces the following equation
$$|\E(x_{i},x_{i})| = \sum_{j} |\B(px_{i},px_{i})_{j}| \cdot |F_{j}|,$$
where $F_{j}$ is the fiber of $p : \E(x_{i},x_{i}) \to \B(px_{i},px_{i})$ over $\mathbf{1} \to \B(px_{i},px_{i})$.
If $e$ of $\pi_{0}\B(px_{i},px_{i})$ denotes the unit, then $F_{e}$ is weakly equivalent to $\F(x_{i},x_{i})$.
Consider the following pullback diagrams
\[
\xymatrix{
F_{j} \ar[r]^{} \ar[d] & \E(x_{i},x_{i})_{j} \ar[r] \ar[d] & \E(x_{i},x_{i}) \ar[d]^{p} \\
\mathbf{1} \ar[r]_{} & \B(px_{i},px_{i})_{j} \ar[r] & \B(px_{i},p_{x_{i}}).
}
\]
Similarly to Lemmas \ref{fiber} and \ref{fiber2}, 
$F_{j}$ is weakly equivalent to $F_{e}$ and $\F(x_{i},x_{i})$ since 
$B(px_{i},px_{i})_{j}$ and $B(px_{i},px_{i})_{e}$ are weakly equivalent to each other by Lemma \ref{component}.
Thus we have the following equation
\begin{equation}
\begin{split}
\chi(\E) &=\sum_{i} \chi(\E_{i}) \\ \notag
&= \sum_{i} (|\E(x_{i},x_{i})|)^{-1} \\
&= \sum_{i} \sum_{j} \left( |\B(px_{i},px_{i})_{j}| \cdot |F_{j}| \right)^{-1} \\
&= \sum_{i} \left( \pi_{0}\B(px_{i},px_{i} )^{\sharp}  \cdot |\B(px_{i},px_{i})_{e}| \cdot |\F(x_{i},x_{i})| \right)^{-1} \\
&= \sum_{i} \left( |\B(px_{i},px_{i})| \cdot |\F(x_{i},x_{i})| \right)^{-1} \\
&= \chi(\B)\chi(\F),
\end{split}
\end{equation}
by Lemma \ref{groupoid}, where $\pi_{0}\B(px_{i},px_{i} )^{\sharp}$ is its cardinal.
\end{proof}

\begin{corollary}\label{product_formula}
Suppose that $\V$ is a monoidal model category with a measure preserving fibrations and it satisfies Assumption \ref{assumption}.
Let $p : \E \to \B$ be a naive fibration between fibrant $\V$-categories admitting Euler characteristic, and also any fiber admits Euler characteristic. 
If both $\pi_{0}\E$ and $\pi_{0}\B$ are groupoids and $\B$ is a finite coproduct of $\coprod_{i} \B_{i}$ for connected $\V$-categories $\B_{i}$, we have 
$$\chi(\E) = \sum_{i} \chi(\B_{i})\chi(\F_{i})$$
where $\F_{i}$ is the fiber of $p$ over an object of $\B_{i}$.
\end{corollary}
\begin{proof}
Let $\E_{i}$ be the full subcategory of $\E$ whose set of objects is the inverse image 
$p^{-1}((\B_{i})_{0})$. Then $\E$ is decomposed as the coproduct $\coprod_{i} \E_{i}$ and
\[
\xymatrix{
\E_{i} \ar[r] \ar[d] & \E \ar[d]^{p} \\
\B_{i} \ar[r] & \B \\
}
\]
is a pullback diagram. 
The left-hand vertical morphism $\E_{i} \to \B_{i}$ is a fibration over the connected base $\B_{i}$ with the fiber $\F_{i}$.
Theorem \ref{product2} shows $\chi(\E_{i})= \chi(\B_{i})\chi(\F_{i})$ and the conclusion 
$\chi(\E) = \sum_{i} \chi(\B_{i})\chi(\F_{i})$.
\end{proof}
\subsection{Examples}

\begin{enumerate}\label{example}
\item Let $(\Set,\times, *)$ be the category of sets equipped with the trivial model structure.
Denote the full subcategory consisting of finite sets by $\mathbf{set}$, and
define a measure $\sharp : \mathbf{set} \to \Z \subset \Q$ by the cardinality of sets.
A category enriched by $\Set$ is a small category. It gives the Euler characteristic of finite categories that coincides with the one defined in \cite{Lei08a}.
The category of small categories admits the canonical model structure coinciding with the folk model structure.
Theorem \ref{product_formula} implies that an isofibration between finite groupoids satisfies the product formula.
\item Let $(\Cat,\times, *)$ be the category of small categories equipped with the folk model structure.
Denote the full subcategory consisting of finite categories admitting Leinster's Euler characteristic \cite{Lei08a} by $\mathbf{cat}$, and define a measure $\chi : \mathbf{cat} \to \Q$ as the Euler characteristic of finite categories. This gives the Euler characteristic of 2-categories.
Theorem \ref{product_formula} implies that a fibration of 2-categories \cite{Lac02} between finite 2-groupoids satisfies the product formula.
A 2-groupoid $G$ is a 2-category all of whose 1-morphisms and 2-morphisms are invertible.
For an object $x$ of $G$, define $\pi_{1}(G,x)$ as $\pi_{0}G(x,x)$ and $\pi_{2}(G,f)$ as the set of 2-morphisms from a 1-morphism $f$ to itself. Lemma \ref{groupoid} implies that 
$$\chi(G)= \sum_{[x] \in \pi_{0}G}\left( \sum_{[y] \in \pi_{1}(G,x)} (\pi_{2}(G,y)^{\sharp})^{-1}\right)^{-1},$$
where $(-)^{\sharp}$ is its cardinal.
Furthermore, Lemma \ref{component} shows that $\pi_{2}(G,y)^{\sharp} = \pi_{2}(G,1_{x})^{\sharp}$ for any element $[y]$ of $\pi_{1}(G,x)$.
Hence,
$$\chi(G)= \sum_{[x] \in \pi_{0}G} \frac{\pi_{2}(G,1_{x})^{\sharp}}{\pi_{1}(G,x)^{\sharp}}.$$
\item Let $(\Top, \times, *)$ be the category of topological categories equipped with the classical model structure.
Denote the full subcategory consisting of spaces having the homotopy type of a finite CW complex by $\mathbf{cw}$, and
define a measure $\chi : \mathbf{cw} \to \Z \subset \Q$ by the topological Euler characteristic. This gives the Euler characteristic of topological categories. 
We will examine this case deeply in the next section.
\end{enumerate}

\section{The classifying spaces of topological categories}

\subsection{The geometric realization of simplicial spaces}
The classifying space of a topological category is introduced in \cite{Seg68}.
It is defined as the geometric realization of a simplicial space called the nerve.

\begin{definition}
The category $\Delta$ consists of totally ordered sets 
$$[n]=\{0< 1 < 2 < \dots <n\}$$
for $n \geqq 0$ as objects and order preserving maps between them as morphisms.
A {\em simplicial space} is a functor from $\Delta^{\op}$ to the category $\Top$ of spaces
$$\Delta^{\op} \longrightarrow \Top.$$
For a simplicial space $X$ and an order preserving map $\varphi : [n] \to [m]$, 
let $\varphi_{*}$ denote $X(\varphi) : X_{m} \to X_{n}$.
The category $\Top^{\Delta^{\op}}$ of simplicial spaces is defined as the functor category from $\Delta^{\op}$ to $\Top$.
Let $\Delta_{+}$ denote the subcategory of $\Delta$ having the same objects as $\Delta$ 
and injective order preserving maps as morphisms. If $n>m$, there is no morphism $[n] \to [m]$ in $\Delta_{+}$.
A {\em $\Delta$-space} is a functor 
$$\Delta_{+}^{\op} \longrightarrow \Top.$$
Let $\Top^{\Delta^{\op}_{+}}$ denote the category of $\Delta$-spaces.
The canonical inclusion functor $\Delta_{+} \to \Delta$ induces a functor
$$\flat : \Top^{\Delta^{\op}} \longrightarrow \Top^{\Delta^{\op}_{+}}.$$
\end{definition}

A simplicial space can be written as a sequence of spaces 
$X_{n}$ equipped with face maps $d_{j} : X_{n} \to X_{n-1}$ and degeneracy maps $s_{i} : X_{n-1} \to X_{n}$ 
satisfying the simplicial identities (see \cite{May92}). 
Similarly, a $\Delta$-space is a sequence of spaces equipped with only face maps.
The above functor $\flat$ makes simplicial spaces forget their degeneracy maps.

\begin{definition}
A cosimplicial space is a functor $\Delta \to \Top$.
The standard cosimplcial space is defined by the standard $n$-simplex $\Delta^{n}$ for $n \geqq 0$ and maps
$\varphi^{*} : \Delta^{n} \to \Delta^{m}$ for each $\varphi : [n] \to [m]$ given by the linear extension of 
the map $v_{i} \mapsto v_{\varphi(i)}$ on vertices of standard simplices.
\end{definition}

\begin{definition}[Geometric realization]
Let $X$ be a simplicial space.
The {\em geometric realization} $|X|$ of $X$ is the space defined by
$$|X| = \left( \coprod_{n \geqq 0} \Delta^{n} \times X_{n} \right) / (t,\varphi_{*}(x)) \sim (\varphi^{*}(t),x) $$ 
for all order preserving maps $\varphi : [n] \to [m]$ and points $x$ in $X_{m}$ and $t$ in $\Delta^{n}$.
Similarly, the \textit{geometric realization} of a $\Delta$-space $Y$ is defined by 
$$||Y|| = \left( \coprod_{n \geqq 0} \Delta^{n} \times Y_{n} \right)/(t,\varphi_{*}(y)) \sim (\varphi^{*}(t),y) $$ 
for all injective order preserving maps $\varphi : [n] \to [m]$ and points $y$ in $Y_{m}$ and $t$ in $\Delta^{n}$.
The {\em fat realization} $||X||$ of a simplicial space $X$ is defined as the realization $||X^{\flat}||$ 
of the $\Delta$-space $X^{\flat}$.
We obtain the canonical projection $||X|| \to |X|$ by taking the quotient of the degenerate part.
\end{definition}

The fat realization is introduced by Segal in \cite{Seg74}.
Compared with the normal geometric realization, the fat realization is easy to treat in homotopy theory.
The following two properties do not hold in the general case of the normal geometric realization.

\begin{proposition}[Proposition A.1 of \cite{Seg74}]\ 
\begin{enumerate}
\item Let $f : X \to Y$ be a map of simplicial spaces which is degreewise a homotopy equivalence. Then, the induced map 
$||f|| : ||X|| \to ||Y||$ is a homotopy equivalence.
\item If $X$ is a simplicial space which is degreewise of the homotopy type of a CW-complex, then so is $||X||$.
\end{enumerate}
\end{proposition}

For a simplicial space $X$ and a nonnegative integer $m$, denote
$$||X||^{(m)} = \left( \coprod_{0 \leqq n \leqq m} \Delta^{n} \times X_{n} \right)/(t,\varphi_{*}(x)) \sim (\varphi^{*}(t),x).$$
We have the following pushout diagram
\[
\xymatrix{
\partial \Delta^{n} \times X_{n} \ar[r] \ar[d] & ||X||^{(n-1)} \ar[d] \\
\Delta^{n} \times X_{n} \ar[r] & ||X||^{(n)}
}
\]
and the fat realization $||X||$ is the following sequential colimit
$$||X||^{(0)} \subset ||X||^{(1)} \subset \dots \subset ||X||^{(n)} \subset \cdots .$$
This construction is a key point of the proof of the above proposition in \cite{Seg74}.
Since the geometric realization of a $\Delta$-space has the same construction, 
it has the same properties as the fat realization.

\begin{corollary}\label{delta} \ 
\begin{enumerate}
\item Let $f : X \to Y$ be a map of $\Delta$-spaces which is degreewise a homotopy equivalence. Then, the induced map 
$||f|| : ||X|| \to ||Y||$ is a homotopy equivalence.
\item If $X$ is a $\Delta$-space which is degreewise of the homotopy type of a CW-complex, then so is $||X||$.
\end{enumerate}
\end{corollary}

Regarding the second property above, 
we observe more explicit cell structure of $||X||$.
We regard the standard simplex $\Delta^{n}$ 
as a natural CW-complex, and denote the $n$-face $\mathrm{Int}(\Delta^{n})$ by $\tau^{n}$.

\begin{definition}
Let $X$ be a cell complex. The set of faces of $X$ is denoted by $P(X)$.
We give a partial order on $P(X)$ defined by $e_{\lambda} \leqq e_{\mu}$ if $ e_{\lambda} \subset \overline{e_{\mu}}$ for any $e_{\lambda}$ and $e_{\mu}$ of $P(X)$.
We call $P(X)$ the \textit{face poset} of $X$.
\end{definition}

\begin{proposition}\label{homotopy_type}
Let $X$ be a $\Delta$-space which is degreewise of the homotopy type of a CW-complex. 
Then, $||X||$ has the homotopy type of a CW-complex consisting of cells
$\tau^{n} \times \sigma$ for $\sigma \in P(X_{n})$ and $n \geqq 0$.
\end{proposition}
\begin{proof}
We use an induction on $m$ for $||X||^{(m)}$.
When $m=0$, the space $||X||^{(0)}=\Delta^{0} \times X_{0}$ 
is a CW-complex consisting of $\tau^{0} \times \sigma$ for $\sigma \in P(X_{0})$.
Assume that $||X||^{(m-1)}$ is a CW-complex consisting of 
$\tau^{n} \times \sigma$ for $\sigma \in P(X_{n})$ and $0 \leqq n \leqq m-1$.
Consider the following push-out diagram 
\[
\xymatrix{
\partial \Delta^{m} \times X_{m} \ar[r]^{f} \ar[d]_{} & ||X||^{(m-1)} \ar[d] \\
\Delta^{m} \times X_{m} \ar[r] & ||X||^{(m)}.
}
\]
We may assume $f$ to be a cellular map by the cellular approximation theorem.
Since $\Delta^{m} \times X_{m}$ is the product CW-complex and 
$\partial \Delta^{m} \times X_{m}$ is a subcomplex, 
the above push-out diagram implies that $||X||^{m}$ is a CW-complex consisting of 
$\tau^{n} \times \sigma$ for an element $\sigma$ of $P(X_{n})$ and $0 \leqq n \leqq m$.
\end{proof}

\subsection{The Euler characteristic of acyclic categories}

In this subsection, we show that 
the Euler characteristic of an acyclic topological category $\A$ 
coincides with that of the classifying space $B\A$ of $\A$.
Recall that the category of topological spaces is 
equipped with the monoidal model structure and the measure $\chi : \mathbf{cw} \to \Z$ in Example \ref{example}. We use this measure in the rest of this paper.

\begin{definition}
Let $\T$ be a topological category. We say that $\T$ is {\em acyclic} if
each endomorphism space $\T(x,x)$ is a single point for any object $x$ of $\T$ 
and for objects $y$ and $z$ of $\T$ such that $y\not =z$, if $\T(y,z) \not=\phi$, then $\T(z,y) = \phi$.  
\end{definition}

\begin{definition}[Classifying space]
Let $\T$ be a topological category. The \textit{nerve} $N\T$ is a simplicial space defined by 
$$N_{n}\T = \coprod_{x_{i} \in \mathcal{T}_{0}} \T(x_{n},x_{n-1}) \times \T(x_{n-1},x_{n-2}) \times \dots \times \T(x_{1},x_{0}).$$
The face map $d_{j} : N_{n}\T \to N_{n-1}\T$ is given by composing or removing morphisms 
$$
d_{j}(f_{n},\dots,f_{1}) = \begin{cases}
(f_{n},\dots,f_{2}) \ \ & j=0,\\
(f_{n},\dots,f_{j+1} \circ f_{j}, \dots, f_{1}) \ \ & 0<j<n,\\
(f_{n-1},\dots,f_{1}) \ \ & j=n,\\
\end{cases}
$$
and the degeneracy map $s_{i} : N_{n}\T \to N_{n+1}\T$ is given by inserting identity morphism 
$$s_{i}(f_{n},\dots,f_{1}) = (f_{n},\dots, f_{i}, 1, f_{i-1}, \dots f_{1}).$$
The classifying space $B\T$ of $\T$ is defined as 
the geometric realization $|N\T|$ of the nerve $N\T$.
\end{definition}

\begin{definition}
Let $\A$ be an acyclic topological category.
The non-degenerate nerve $\overline{N}\A$ is a $\Delta$-space defined by 
$$\overline{N}_{n}\A = \coprod_{x_{i} \not= x_{i-1}} \T(x_{n},x_{n-1}) \times \T(x_{n-1},x_{n-2}) \times \dots \times \T(x_{1},x_{0}).$$
The face map $d_{j} : \overline{N}_{n}\A \to \overline{N}_{n-1}\A$ 
is given similar to the ordinary nerve $N\A$.
\end{definition}

\begin{theorem}[Lemma B.13 in \cite{Tama}]\label{realization}
If $\A$ is a finite acyclic topological category, 
then the classifying space $B\A$ is homeomorphic to $||\overline{N}\A||$.
\end{theorem}

Let $\mathcal{A}$ be a finite measurable acyclic topological category. 
We give a partial order on the set of objects $\A_{0}$ such that $a \leqq b$ if $\A(a,b)$ is not empty.
For simplicity, let $\chi_{a,b}$ denote the Euler characteristic of $\A(a,b)$. 
If $a \not \leqq b$, we have $\chi_{a,b}=0$ since $\A(a,b)$ is empty.

\begin{lemma}\label{chi_T}
The Euler characteristic of $\mathcal{A}$ is
$$\chi(\mathcal{A})=\sum^{\infty}_{j=0} \sum_{a_0< \cdots <a_j} (-1)^j \chi_{a_{j-1},a_{j}}\cdots \chi_{a_0,a_{1}}. $$
\end{lemma}
\begin{proof}
Define $v:\mathcal{A}_0\to \mathbb{Z}\subset \mathbb{Q}$ by  
 $$v(a)=\sum^{\infty}_{j=0} \sum_{a_0<\cdots<a_{j-1}<a_{j}=a} (-1)^j \chi_{a_{j-1},a} \cdots \chi_{a_0,a_{1}} $$ 
for any $a \in \A_{0}$. Let $\ell(a)$ denote the length of $a \in \A_{0}$ defined as the maximum length of sequences 
$a_{0}< a_{1} < \cdots <a_{k}=a$ ending at $a$.
For an object $b \in \A_{0}$, since $\chi_{b,b}=1$, the following equation holds:
\begin{eqnarray*}
\sum_{a \in \A_{0}}v(a) \xi_{\A}(a,b)
&=&\sum_{a \in \A_{0}} \sum_{j=0}^{\infty} \sum_{a_0<\cdots<a_{j-1}<a_{j}=a} (-1)^{j}\chi_{a,b}\chi_{a_{j-1},a} \cdots \chi_{a_{0},a_{1}} \\ \notag
&=&\sum_{j=0}^{\infty} \sum_{a_{0}<\cdots<a_{j-1}<a_{j} \leqq b} (-1)^{j}\chi_{a,b}\chi_{a_{j-1},a} \cdots \chi_{a_{0},a_{1}} \\
&=&\sum_{a_{0} \leqq b} \chi_{a_{0},b}+(-1)\sum_{a_{0}<a_{1} \leqq b}\chi_{a_{1},b}\chi_{a_{0},a_{1}} + \cdots \\
&&+(-1)^{\ell(b)} \sum_{a_{0}< \cdots <a_{\ell(b)}=b} \chi_{a_{\ell(b)},b}\cdots \chi_{a_{0},a_{1}} \\
&=&\chi_{b,b}=1.
\end{eqnarray*}
Hence, $v$ is a coweighting on $\A$, and we have $$\chi(\mathcal{A})=\sum_{a \in \A_{0}} v(a)= \sum^{\infty}_{j=0} \sum_{a_0<\cdots<a_j} (-1)^j \chi_{a_{j-1},a_{j}}\cdots \chi_{a_0,a_{1}}.$$
\end{proof}

\begin{theorem}\label{classifying}
The Euler characteristic $\chi(\A)$ coincides with the Euler characteristic $\chi(B\A)$ of the classifying space $B\A$.
\end{theorem}
\begin{proof}
By Theorem \ref{realization} and Proposition \ref{homotopy_type}, 
the classifying space $B\A$ has the homotopy type of a finite CW-complex 
whose $n$-cell corresponds to $\tau^{j} \times \sigma^{k_{j}} \times \dots \times  \sigma^{k_{1}}$
for $\tau^{j}=\mathrm{Int}(\Delta^{j})$ and a $k_{i}$-cell $\sigma^{k_{i}}$ of $\A(a_{i},a_{i+1})$ 
such that $n=j+\sum_{i=1}^{j} k_{i}$ and $a_{i} \not= a_{i+1}$ for any $i$.
Let $\A_{a,b}^{(k)}$ denote the number of $k$-cells of $\A(a,b)$.
The Euler characteristic $\chi(B\A)$ can be computed as the alternating sum of the numbers of cells as follows:
\begin{eqnarray*}
\chi(B\mathcal{A})&=&\sum_{n=0}^{\infty} (-1)^n \left(\text{the number of $n$-cells of $B\mathcal{A}$ }\right) \\
&=& \sum_{n=0}^{\infty}(-1)^n\sum_{j=0}^{\infty} \sum_{j+k_1+\cdots+k_j=n} \sum_{a_0<\cdots<a_j} \mathcal{A}_{a_{j-1},a_{j}}^{(k_j)} \cdots \mathcal{A}_{a_{0},a_{1}}^{(k_1)} \\
&=& \sum_{n=0}^{\infty} \sum_{j=0}^{\infty}  \sum_{j+k_1+\cdots+k_j=n} \sum_{a_0<\cdots<a_j} (-1)^j(-1)^{k_j} \mathcal{A}_{a_{j-1},a_{j}}^{(k_j)} \cdots (-1)^{k_1}\mathcal{A}_{a_{0},a_{1}}^{(k_1)} \\
&=& \sum_{j=0}^{\infty} \sum_{1\le \ell \le j} \sum_{k_{\ell}=0}^{\infty} \sum_{a_0<\cdots<a_j} (-1)^{j}(-1)^{k_j} \mathcal{A}_{a_{j-1},a_{j}}^{(k_j)} \cdots (-1)^{k_1}\mathcal{A}_{a_{0},a_{1}}^{(k_1)} \\
&=&\sum^{\infty}_{j=0} \sum_{a_0<\cdots<a_j} (-1)^j \chi_{a_{j-1},a_{j}}\cdots \chi_{a_0,a_{1}}.
\end{eqnarray*}
Consequently we have $\chi(B\A)=\chi(\A)$ by Lemma \ref{chi_T}.
\end{proof}

\section{The Euler characteristic of a cellular stratified space}

A cellular stratified space is a generalization of a cell complex introduced by Tamaki in \cite{Tama}.
It is a space of attached cells with possibly incomplete boundaries. 
He introduces many examples of cellular stratified spaces,
for instance, regular cell complexes, complements of complexified hyperplane arrangements, 
and configuration spaces of graphs and spheres \cite{Tama}, \cite{Tamb}, \cite{Tamc}.

\begin{definition}[Stratified space]\label{def1}
A {\em stratification} on a Hausdorff space $X$ indexed
by a poset $\Lambda$ is a map $\pi : X \to \Lambda$ satisfying the following properties:
\begin{enumerate}
\item Each $\pi^{-1}(\lambda)$ is connected, and open in $\overline{\pi^{-1}(\lambda)}$.
\item $\pi^{-1}(\lambda) \subset \overline{\pi^{-1}(\mu)}$ if and only if $\lambda \leqq \mu$.
\end{enumerate}
For simplicity, we denote $e_{\lambda} = \pi^{-1}(\lambda)$ and call it a \emph{face}.
The indexing poset $\Lambda$ is called the face poset of $X$ and denoted $P(X)$. 
A \emph{stratified space} $(X,\pi)$ is a pair of a space $X$ and its stratification $\pi$.
When $P(X)$ is a finite poset, we call the stratified space $X$ to be finite. 

We say a face $e_{\lambda}$ is \emph{normal} if $e_{\mu} \subset \overline{e_{\lambda}}$ whenever
$e_{\mu} \cap \overline{e_{\lambda}} \not= \phi$.
When all faces are normal, the stratification is said to be normal.

Let $(X, \pi_{X})$ and $(Y,\pi_{Y})$ be stratified spaces. 
A morphism of stratified spaces $(f, \overline{f})$ is a pair 
of a continuous map $f : X \longrightarrow Y$ and a
map of posets $\overline{f} : P(X) \longrightarrow P(Y)$ making the following diagram commutative
\[
\xymatrix{
X \ar[d]_{\pi_{X}} \ar[r]^{f} & Y \ar[d]^{\pi_{Y}} \\
P(X) \ar[r]_{\overline{f}} & P(Y). }
\]
When $f$ is a homeomorphism, 
the map $(f,\overline{f})$ is called a {\em subdivision}.
\end{definition}

\begin{definition}[Cell structure]\ 
\begin{enumerate}
\item An {\em $n$-globe} is a subspace $D$ of a disk $D^{n}$ with $\mathrm{Int}(D^{n}) \subset D \subset D^{n}$, where $\mathrm{Int}(D^{n})$ is the interior of $D^{n}$.
\item An $n$-cell structure on $e \subset X$ is a quotient map 
$\varphi : D \to \overline{e}$ from an $n$-globe $D$ whose restriction $\varphi |_{\mathrm{Int}(D^{n})} : \mathrm{Int}(D^{n}) \to e$ is a 
homeomorphism. We say that $n$ is the dimension of $e$, and denote it $\dim e$. We say that an $n$-cell is closed when $D$ is a closed disk $D^{n}$.
\item A cellular stratified space is a stratified space whose faces are equipped 
with cell structures such that 
$$\varphi(\partial D_{\lambda}) \subset \bigcup_{\dim e_{\mu}<\dim e_{\lambda}} e_{\mu},$$
where $\partial D_{\lambda}$ is the boundary of $D_{\lambda}$. When all cells are closed (called closed cellular stratified space), this means a cell complex.
Furthermore, a cellular stratified space is called a {\em CW-stratified space} if it satisfies
the closure finite and weak topology conditions (see Definition 2.19 in \cite{Tama}).
\end{enumerate}
\end{definition}

\begin{definition}[Cylindrical structure]\label{cri}
A \emph{cylindrical structure} on a normal cellular stratified $X$ consists of
\begin{itemize}
\item a normal stratification on $\partial D^{n}$ containing 
$\partial D_{\lambda}$ as a stratified subspace for each 
$n$-cell $\varphi_{\lambda} : D_{\lambda} \to \overline{e_{\lambda}}$.
\item a stratified space $P_{\lambda,\mu}$ and a morphism of stratified spaces
$$b_{\lambda, \mu} : P_{\lambda,\mu} \times D_{\lambda} \longrightarrow \partial D_{\mu} \subset D_{\mu}$$
for each $\lambda < \mu$.
\item a morphism of stratified spaces 
$$c_{\lambda_{1},\lambda_{2},\lambda_{3}} : P_{\lambda_{2},\lambda_{3}} \times P_{\lambda_{1},\lambda_{2}} \longrightarrow P_{\lambda_{1},\lambda_{3}}$$
for each sequence $\lambda_{1}<\lambda_{2}<\lambda_{3}$.
\end{itemize}
satisfying the following conditions:
\begin{enumerate}
\item The restriction of $b_{\lambda,\mu}$ to $P_{\lambda,\mu} \times \mathrm{Int}(D_{\lambda})$ is an embedding.
\item The following diagram is commutative
\[
\xymatrix{
P_{\lambda,\mu} \times D_{\lambda} \ar[d]_{b_{\lambda,\mu}} \ar[r]^{\ \ \ \mathrm{pr}_{2}} & D_{\lambda} \ar[d]^{\varphi_{\lambda}} \\
D_{\mu} \ar[r]_{\varphi_{\mu}} & X. }
\]
\item The following diagram is commutative
\[
\xymatrix{
P_{\lambda_{2},\lambda_{3}} \times P_{\lambda_{1},\lambda_{2}} \times D_{\lambda_{1}} 
\ar[r]^{\ \ \ 1 \times b_{\lambda_{1},\lambda_{2}}} \ar[d]_{c_{\lambda_{1},\lambda_{2},\lambda_{3}} \times 1 } 
& P_{\lambda_{2},\lambda_{3}} \times D_{\lambda_{2}} \ar[d]^{b_{\lambda_{2},\lambda_{3}}} \\
P_{\lambda_{1},\lambda_{3}} \times D_{\lambda_{1}} \ar[r]_{\ \ b_{\lambda_{1},\lambda_{3}}} & D_{\lambda_{3}}. }
\]
\item The map $c$ satisfies the associativity condition, i.e.,
\[
\xymatrix{
P_{\lambda_{3},\lambda_{4}} \times P_{\lambda_{2},\lambda_{3}} \times P_{\lambda_{1},\lambda_{2}} 
\ar[r]^{\ \ \ 1 \times c_{\lambda_{1},\lambda_{2},\lambda_{3}}} \ar[d]_{c_{\lambda_{2},\lambda_{3},\lambda_{4}} \times 1} 
& P_{\lambda_{3},\lambda_{4}} \times P_{\lambda_{1},\lambda_{3}} \ar[d]^{c_{\lambda_{1},\lambda_{3},\lambda_{4}}} \\
P_{\lambda_{2},\lambda_{4}} \times P_{\lambda_{1},\lambda_{2}} \ar[r]_{\ \ c_{\lambda_{1},\lambda_{2},\lambda_{4}}} & P_{\lambda_{1},\lambda_{4}} 
}
\]
is a commutative diagram.
\item We have 
$$\partial D_{\mu} = \bigcup_{\lambda< \mu} b_{\lambda,\mu}(P_{\lambda,\mu} \times \mathrm{Int}(D_{\lambda}))$$
as a stratified space.
\end{enumerate}
A normal cellular stratified space
equipped with a cylindrical structure is called a \emph{cylindrically normal} cellular stratified space.
\end{definition}

\begin{definition}
For a cylindrically normal cellular stratified space $X$,
define a topological category $C(X)$ as follows.
The set of objects $C(X)_{0}$ is the set $P(X)$ of faces. The space of morphisms is defined by
$$C(X)(\lambda,\mu) = P_{\lambda,\mu}$$
for each $\lambda < \mu$ in $P(X)$ with the composition $c$.
Note that $C(X)$ is an acyclic topological category.
\end{definition}

Tamaki constructs an embedding $BC(X) \hookrightarrow X$ from the classifying space of $C(X)$ to the original 
cylindrically normal cellular stratified space $X$, 
and tries to show that it embeds $BC(X)$ as a deformation retract of $X$.
Note that he considers the more general ``stellar" stratified spaces rather than cellular stratified spaces in his paper \cite{Tama}.

\begin{theorem}[Theorem 4.15 in \cite{Tama}]
There exists an embedding $BC(X) \hookrightarrow X$ for a cylindrically normal cellular stratified space $X$.
Furthermore, when all cells are closed, the above embedding is a homeomorphism. 
\end{theorem}

In order to show that the above embedding is a homotopy equivalence for a general cylindrical cellular stratified spaces, 
Tamaki considers the following condition. See section 3.3 in \cite{Tama} for details.

\begin{definition}[Local polyhedral structure]
A {\em locally polyhedral cellular stratified space} consists of 
\begin{itemize}
\item a cylindrically normal CW-stratified space $X$,
\item a family of Euclidean polyhedral complexes $\tilde{F}_{\lambda}$ indexed by $\lambda$ of $P(X)$
\item a family of homeomorphisms $\alpha_{\lambda} : \tilde{F}_{\lambda} \to D^{\dim e_{\lambda}}$ for any $\lambda$ of $P(X)$
\end{itemize}
satisfying the following conditions:
\begin{enumerate}
\item The homeomorphism $\alpha : \tilde{F}_{\lambda} \to D^{\dim e_{\lambda}}$ is a subdivision of the stratified space, where the stratification 
of $D^{\dim e_{\lambda}}$ is defined by the cylindrical structure.
\item The parameter space $P_{\lambda,\mu}$ is a locally cone-like space and the composition 
$$P_{\lambda,\mu} \times F_{\lambda} \stackrel{1 \times \alpha_{\lambda}}{\longrightarrow} P_{\lambda,\mu} \times D_{\lambda} 
\stackrel{b_{\lambda,\mu}}{\longrightarrow} D_{\mu} \stackrel{\alpha^{-1}_{\mu}}{\longrightarrow} F_{\mu}$$
is a PL map, where $F_{\lambda}=\alpha^{-1}(D_{\lambda})$.
\end{enumerate}
Each $\alpha_{\lambda}$ is called a {\em polyhedral replacement} of the cell structure $e_{\lambda}$.
\end{definition}

\begin{theorem}[Theorem 4.16 in \cite{Tama}]
If $X$ is a locally polyhedral cellular stratified space, 
then there exists an embedding $BC(X) \hookrightarrow X$ whose image is a deformation retract of $X$. 
\end{theorem}

\begin{corollary}\label{homotopy_equ}
If $X$ is a locally polyhedral cellular stratified space, 
then the classifying space $BC(X)$ of the cylindrical face category $C(X)$ of $X$ is homotopy equivalent to $X$.
\end{corollary}

\begin{theorem}\label{facecat}
Let $X$ be a finite locally polyhedral cellular stratified space. 
If each parameter space $P_{\lambda,\mu}$ belongs to $\mathbf{cw}$ for $\lambda < \mu$ in $P(X)$, then the Euler characteristic $\chi(X)$ of $X$ is equal to 
$\chi(C(X))$ of the cylindrical face category.
\end{theorem}
\begin{proof}
It follows from Theorem \ref{classifying} and Corollary \ref{homotopy_equ} that 
$$\chi(X) = \chi(BC(X)) = \chi(C(X)).$$
\end{proof}

\begin{example}
Example 4.7 in \cite{Tama} shows that a complex projective space with the minimal decomposition 
$\mathbb{C}\mathbf{P}^{n}=e^{0} \cup e^{2} \cup \cdots \cup e^{2n}$ is a closed cylindrical cellular stratified space 
whose face category $C(\mathbb{C}\mathbf{P}^{n})$ forms 
\[
\xymatrix{
{\bullet} \ar[r]^{S^{1}} & {\bullet} \ar[r]^{S^{1}} & \bullet \ar[r]^{S^{1}} &  \cdots \ar[r]^{S^{1}} & \bullet. 
}
\]
Since $\chi(S^{1})=0$, the similarity matrix $\xi$ of $C(\mathbb{C}\mathbf{P}^{n})$ is the unit matrix with dimension $n+1$.
We can take a weighting $w$ as $w(e^{2i})=1$ for all $0 \leqq i \leqq n$. 
Hence, $\chi(\mathbb{C}\mathbf{P}^{n}) = \chi \left(C \left(\mathbb{C}\mathbf{P}^{n} \right)\right) =n+1$. 
\end{example}

However, for obtaining $\chi(\mathbb{C}\mathbf{P}^{n})$, it is easier to calculate the alternating sum of the numbers of cells than the above.
The next example is non-closed cellular stratified space.
In the following case, 
it is difficult to calculate the Euler characteristic from the numbers of cells.

\begin{example}
Let us consider the space $X=S^{n} \times S^{m}-\{*\}$ with one point removed from the product of spheres.
This is a cellular stratified space consisting of the cell structure
$$\varphi_{j} : D_{j}=\mathrm{Int}(D^{j}) \longrightarrow S^{j}-\{*\},$$
which is the restriction of the canonical projection $D^{j} \to S^{j}$ collapsing $\partial D^{j}$ to a single point
for $j=n,m$, and
$$\varphi_{n+m}=\varphi_{n} \times \varphi_{m} : D_{n+m}=(D^{n} \times D^{m})-(\partial D^{n} \times \partial D^{m}) \longrightarrow S^{n} \times S^{m}-\{*\}$$
where we regard $D^{n} \times D^{m}$ as $D^{n+m}$.
The boundary of $D_{n+m}$ is
\begin{equation}
\begin{split}
\partial D_{n+m} &=\partial(D^{n} \times D^{m})-(\partial D^{n} \times \partial D^{m}) \notag \\
&= (\partial D^{n} \times D^{m}) \cup (D^{n} \times \partial D^{m}) - (\partial D^{n} \times \partial D^{m}) \\
&= \left( \partial D^{n} \times \mathrm{Int}(D^{m}) \right) \coprod \left( \mathrm{Int}(D^{n}) \times \partial D^{m} \right)
\end{split}
\end{equation}
A normal stratification on $\partial D^{n+m}$ is induced by the canonical cell decomposition on $\partial I^{n+m}$
and $\partial I^{n+m} \cong \partial D^{n+m}$.
The cylindrical structure is given by the inclusions
$$b_{ij} : D_{i} \times S^{j-1}=\mathrm{Int}(D^{i}) \times S^{j-1} \hookrightarrow \partial D_{n+m}$$
for $(i,j)=(n,m), (m,n)$. Then the cylindrical face category $C(X)$ forms 
\[
\xymatrix{
{\bullet} \ar[r]^{S^{n-1}} & {\bullet}  & \bullet \ar[l]_{S^{m-1}}. 
}
\]
A family of polyhedral replacements $\{\alpha_{j} : I^{j} \cong D^{j}\}_{j=n,m,m+n}$ makes $X$ a locally polyhedral cellular stratified space.
The similarity matrix of $C(X)$
\begin{equation}
\xi=
\begin{pmatrix}
1 & 0 & \chi(S^{n-1}) \\
0 & 1 & \chi(S^{m-1}) \\
0 & 0 & 1
\end{pmatrix} \notag
\end{equation}
has an inverse matrix
\begin{equation}
\xi^{-1}=
\begin{pmatrix}
1 & 0 & 0 \\
0 & 1 & 0 \\
-\chi(S^{n-1}) & -\chi(S^{m-1}) & 1
\end{pmatrix} \notag
.
\end{equation}
Theorem \ref{facecat} shows that the Euler characteristic of $X$ is 
$$\chi(X)=\chi(C(X)) = 3-\chi(S^{n-1})-\chi(S^{m-1})=\begin{cases}
-1 &\ \ \textrm{if both $n$ and $m$ are odd,}\\
3 &\ \ \textrm{if both $n$ and $m$ are even,}\\
1 &\ \ \textrm{otherwise.}\\
\end{cases}
$$
Indeed, the space $X=S^{n} \times S^{m}-\{*\}$ is homotopy equivalent to $BC(X) = S^{n-1} \vee S^{m-1}$.
The Euler characteristic $\chi(X)$ can also be obtained by calculating $H_{*}(S^{n-1} \vee S^{m-1})$.
\end{example}

\appendix

\section{Homotopy pullbacks}

Let us recall some fundamental properties of a model category $\M$.

\begin{definition}
A model category $\M$ is called {\em right proper}
if the pullback of a weak equivalence along a fibration is a weak equivalence.
\end{definition}

\begin{proposition}[Proposition 13.1.2 in \cite{Hir03}]
Suppose that $f : A \to B$ is a weak equivalence and $p : E \to B$ is a fibration in a model category $\M$.
If both $A$ and $B$ are fibrant, the pull back $A \times_{B} E \to E$ of $f$ along $p$ is a weak equivalence.
\end{proposition}

\begin{corollary}\label{right}
If every object of a model category $\M$ is fibrant, then $\M$ is right proper.
\end{corollary}

The following proposition is proved in the case of a right proper model category in Proposition 13.3.10 of \cite{Hir03}.
However, it holds for fibrant objects instead of right properness by Corollary \ref{right}.

\begin{proposition}\label{stable}
For the following commutative diagrams
\[
\xymatrix{
A \ar[r] \ar[d] & C \ar[d] & B \ar[l] \ar[d] \\
A' \ar[r] & C' & B' \ar[l]
}
\]
in a model category $\M$, let the above three vertical morphisms be weak equivalences, and let all objects be fibrant.
If at least one morphism in each row is a fibration, then the induced morphism between pullbacks $A \times_{C} B \to A' \times_{C'} B'$ is a weak equivalence.
\end{proposition}

The ordinary pull back in a model category is not a homotopy invariant in general.
The notion of homotopy pullbacks fixes this inconvenient property.

\begin{definition}
A {\em path object} $Y^{I}$ of an object $Y$ in a model category $\M$ is a factorization
$$\Delta : Y \stackrel{s}{\longrightarrow} Y^{I} \stackrel{(p_{0},p_{1})}{\longrightarrow} Y \times Y$$
of the diagonal morphism such that $s$ is a weak equivalence and $(p_{0},p_{1})$ is a fibration in $\M$.

For two morphisms $f,g : X \to Y$ in a model category $\M$, a {\em right homotopy} from $f$ to $g$ consists of 
a path object $Y^{I}$ and a morphism $H : X \to Y^{I}$ such that $p_{0}H=f$ and $p_{1}H=g$.
Then we say that $f$ is right homotopic to $g$.
\end{definition}

The dual notions of {\em cylinder objects} and {\em left homotopic} exist.
If $f$ is both left homotopic and right homotopic to $g$, then we say $f$ is {\em homotopic} to $g$.

\begin{definition}[Homotopy pullback]\label{hopu}
For a diagram $A \stackrel{f}{\rightarrow} C \stackrel{g}{\leftarrow} B$ of fibrant objects in a model category $\M$,
choose a path object $C^{I}$ of $C$. 
The {\em homotopy pullback} $A \times_{C}^{h} B$ of the diagram is the ordinary pullback of the following diagram
\[
\xymatrix{
A \times_{C}^{h} B \ar[r] \ar[d] & A \times B \ar[d]^{f \times g} \\
C^{I} \ar[r]_{(p_{0},p_{1})} & C \times C.
}
\]
The above pullback $A \times_{C}^{h} B$ is isomorphic to the following limit
\[
\xymatrix{
A \times_{C}^{h} B \ar[dd] \ar[dr] \ar[rr] & & B \ar[d]^{g} \\
& C^{I} \ar[r]^{p_{1}} \ar[d]^{p_{0}} & C \\
A \ar[r]_{f} & C. 
}
\]
\end{definition}

Brown introduces the notion of homotopy pullbacks in a category equipped with only weak equivalences and fibrations \cite{Bro73}. 
Hirschhorn also gives a different definition of homotopy pullbacks in section 13.3 of \cite{Hir03}. 
However, we use the above Definition \ref{hopu} in order to investigate the relation with homotopy commutativity of diagrams.

\begin{lemma}
The weak equivalence class of the homotopy pullback $A \times_{C}^{h} B$
does not depend on the choice of path object $C^{I}$ of $C$. 
\end{lemma}
\begin{proof}
We can find a path object $C^{I}$ of $C$ such that $s : C \to C^{I}$ is a trivial cofibration by the factorization axiom.
Fix such a path object $C^{I}$ of $C$. 
For any path object $(C^{I})'$ of $C$, the lifting axiom induces a weak equivalence $C^{I} \to (C^{I})'$ making the following diagram commute
\[
\xymatrix{
C \ar[r]^{\sim} \ar[d]_{s}^{\sim} & (C^{I})' \ar[d]^{(p_{0}',p_{1}')} \\
C^{I} \ar[r]_{(p_{0},p_{1})} \ar@{..>}[ur] & C \times C.
}
\]
The following commutative diagram
\[
\xymatrix{
C^{I} \ar[r] \ar[d]_{\sim} & C \times C \ar@{=}[d] & A \times B \ar[l] \ar@{=}[d] \\
(C^{I})' \ar[r] & C \times C  & A \times B \ar[l]
}
\]
induces a weak equivalence between pullbacks of the row diagrams by Proposition \ref{stable}. Hence, $A \times_{C}^{h} B$ is determined uniquely up to weak  equivalence not depending on the choice of the path object.
\end{proof}

The homotopy pullback $A \times_{C}^{h} B$ makes the diagram
\[
\xymatrix{
A \times_{C} B \ar[r] \ar[d] & B \ar[d] \\
A \ar[r] & C
}
\]
commute up to homotopy.
The following proposition is proved in the dual statement 
using homotopy pushouts (double mapping cylinder) in Theorem 2.16 of \cite{KP97}.

\begin{theorem}\label{unique_mor}
For a homotopy commutative diagram 
\[
\xymatrix{
X \ar[r]^{f} \ar[d]_{q} & B \ar[d]^{p} \\
A \ar[r]_{g} & C 
}
\]
of fibrant objects in a model category $\M$, there exists a unique morphism 
$A \times_{C}^{h} B \to X$ making the diagram commute up to homotopy.
\end{theorem}
 
The following pullback diagram
\[
\xymatrix{
A \times_{C} B \ar[r] \ar[d] & A \times B \ar[d]^{f \times g} \\
C \ar[r]_{\Delta} & C \times C.
}
\]
induces the canonical morphism $A \times_{C} B \to A \times_{C}^{h} B$ from the ordinary pullback to the homotopy pullback.

\begin{theorem}\label{back}
Let $A \stackrel{f}{\rightarrow} C \stackrel{g}{\leftarrow} B$ be a diagram of fibrant objects in a model category $\M$.
If $f$ or $g$ is a fibration, then the canonical morphism $\alpha : A \times_{C} B \to A \times_{C}^{h} B$ is a weak equivalence.
\end{theorem}
\begin{proof}
Let $f$ be a fibration in $\M$. The following pullback diagrams
\[
\xymatrix{
A \times_{C} B \ar[r]^{\alpha} \ar[d] & A \times_{C}^{h} B \ar@{->}'[d][dd]  \ar[rr] & & B \ar[dd]^{g} \ar[dl]_{g} \\
A \ar@{=}[dd] \ar[rr]^{} \ar[dr]^{\sim}& & C \ar[d]_{s} \ar@{=}[dr] & \\
& A \times_{C} C^{I} \ar[d]_{\sim} \ar[r] & C^{I} \ar[r]^{p_{1}} \ar[d]^{p_{0}}_{\sim} & C \\
A \ar@{=}[r] &A \ar[r]_{f} & C. 
}
\]
imply that the canonical morphism $\alpha : A \times_{C} B \to A \times_{C}^{h} B$ is a weak equivalence by Proposition \ref{stable}.
\end{proof}

\begin{lemma}\label{homotopic}
Let $f,f' : A \to B$ and $g : B \to C$ be morphisms in a model category $\M$. 
If $f$ is homotopic to $f'$, then the homotopy pullback along $f$ is weakly equivalent to the one along $f'$.
\end{lemma}
\begin{proof}
Let $A \wedge I$ be a cylinder object and $h : A \wedge I \to C$ be a left homotopy from $f$ to $f'$.
We may take $A \wedge I$ as a fibrant object by Proposition 7.3.4 in \cite{Hir03}.
Consider the following pull back diagrams
\[
\xymatrix{
A \times_{C}^{h} B \ar[r]^{} \ar[d] & (A \wedge I) \times_{C}^{h} B \ar[d] \ar[r] & C^{I} \ar[d] \\ 
A \times B \ar[r]_{i_{j} \times 1}^{\sim} & (A \wedge I) \times B \ar[r] & C \times C
}
\]
for $j=0,1$. The middle vertical morphism $(A \wedge I) \times_{C}^{h} B \to (A \wedge I) \times B$ is a fibration since $C^{I} \to C \times C$ is a fibration.
Proposition \ref{right} implies that the homotopy pullback $A \times_{C}^{h} B$ is weakly equivalent to $(A \wedge I) \times_{C}^{h} B$ for each $j=0,1$.
\end{proof}

\begin{theorem}\label{stable2}
For the following homotopy commutative diagrams
\[
\xymatrix{
A \ar[r]^{f} \ar[d]_{\alpha} & C \ar[d]_{\beta} & B \ar[l]_{g} \ar[d]^{\gamma} \\
A' \ar[r]_{f'} & C' & B' \ar[l]^{g'}
}
\]
in a model category $\M$, let the above three vertical morphisms be weak equivalences, and let all objects be fibrant.
The homotopy pullbacks $A \times_{C}^{h} B $ and $A' \times_{C'}^{h} B'$ are weakly equivalent.
\end{theorem}
\begin{proof}
The following pullback diagrams 
\[
\xymatrix{
A \times_{C'}^{h} B \ar[r]^{} \ar[d] & A' \times_{C'}^{h} B' \ar[d] \ar[r] & (C')^{I} \ar[d] \\ 
A \times B \ar[r]_{\alpha \times \gamma}^{\sim} & A' \times B' \ar[r] & C' \times C'
}
\]
induce that $A \times_{C'}^{h} B \to A' \times_{C'}^{h} B'$ is a weak equivalence.
Lemma \ref{homotopic} implies that the homotopy pull back of 
$A \stackrel{f'\circ \alpha}{\rightarrow} C' \stackrel{g' \circ \gamma}{\leftarrow} B$ 
is weakly equivalent to that of $A \stackrel{\beta \circ f}{\rightarrow} C' \stackrel{\beta \circ g}{\leftarrow} B$.
By taking a path object $C^{I}$ of $C$ such that $s : C \to C^{I}$ is a trivial cofibration, we have a weak equivalence $C^{I} \to (C')^{I}$ making the following diagram commute
\[
\xymatrix{
A \times B \ar[r]^{f \times g} \ar@{=}[d] & C \times C \ar[d]^{\beta \times \beta}_{\sim}  & C^{I} \ar[l]_{(p_{0},p_{1})} \ar[d]_{\sim} \\ 
A \times B \ar[r]_{\beta \circ f \times \beta \circ g} & C' \times C'  & (C')^{I}. \ar[l]^{(p_{0}',p_{1}')}
}
\]
Proposition \ref{stable} implies that $A \times_{C}^{h} B$ is weakly equivalent to $A \times_{C'}^{h} B$.
\end{proof}

\begin{corollary}\label{stable3}
For the following homotopy commutative diagrams
\[
\xymatrix{
A \ar[r]^{f} \ar[d]_{\alpha} & C \ar[d]_{\beta} & B \ar[l]_{g} \ar[d]^{\gamma} \\
A' \ar[r]_{f'} & C' & B' \ar[l]^{g'}
}
\]
in a model category $\M$, let the above three vertical morphisms be weak equivalences, and let all objects be fibrant.
If at least one morphism in each row is a fibration, the pullbacks $A \times_{C} B $ and $A' \times_{C'} B'$ are weakly equivalent.
\end{corollary}
\begin{proof}
This follows from Theorem \ref{stable2} and \ref{back}.
\end{proof}


\begin{thebibliography}{AAA999}
\bibitem [Amr]{Amr} Ilias Amrani. A model structure on the category of topological categories. arXiv:1110.2695v1. 

\bibitem [BD01]{BD01} John C Baez and James Dolan. From finite sets to Feynman diagrams. \textit{Mathematics unlimited--2001 and beyond,} 29--50, Springer, Berlin, 2001.

\bibitem [Ber07]{Ber07} Julia E. Bergner. A model category structure on the category of simplicial categories. \textit{Trans. Amer. Math. Soc.} 359 (2007), no. 5, 2043--2058. 

\bibitem [BL08]{BL08} C. Berger and T. Leinster. The Euler characteristic of a category as the sum of a divergent series, \textit{Homology, Homotopy Appl.}, 10(1): 41-51, 2008.

\bibitem [BM13]{BM13} Clemens Berger and Ieke Moerdijk. On the homotopy theory of enriched categories. \textit{Q. J. Math.} 64 (2013), no. 3, 805--846. 

\bibitem [Bro73]{Bro73} Kenneth S. Brown. Abstract homotopy theory and generalized sheaf cohomology. \textit{Trans. Amer. Math. Soc.} 186 (1973), 419--458.

\bibitem [FLS11]{FLS11} T. M. Fiore, W. L\"{u}ck and R. Sauer. Finiteness obstructions and Euler characteristics of categories. \textit{Adv. Math}, 226(3), 2371--2469, 2011.

\bibitem [Hir03]{Hir03} Philip S. Hirschhorn. Model categories and their localizations. \textit{Mathematical Surveys and Monographs}, 99. American Mathematical Society, Providence, RI, 2003, xvi+457.

\bibitem [Hov99]{Hov99}Mark Hovey. Model categories. \textit{Mathematical Surveys and Monographs}, 63. American Mathematical Society, Providence, RI, 1999. xii+209 pp.

\bibitem [JT91]{JT91}Andre Joyal and Myles Tierney. Strong stacks and classifying spaces. \textit{Category theory (Como, 1990)}, 213--236, Lecture Notes in Math., 1488, Springer, Berlin, 1991.

\bibitem [Kel05]{Kel05}G. M. Kelly. Basic concepts of enriched category theory. \textit{Repr. Theory Appl. Categ.} No. 10 (2005), vi+137 pp.

\bibitem [KP97]{KP97} K. H. Kamps and T. Porter. Abstract homotopy and simple homotopy theory. \textit{World Scientific Publishing Co., Inc., River Edge, NJ}, 1997. x+462 pp.

\bibitem [Lac02]{Lac02}Stephen Lack. A Quillen model structure for 2-categories. \textit{$K$-Theory} 26 (2002), no. 2, 171--205.

\bibitem [Lei08a]{Lei08a}Tom Leinster. The Euler characteristic of a category. \textit{Doc. Math.} 13 (2008), 21--49.

\bibitem [Lei13]{Lei13}Tom Leinster. The magnitude of metric spaces. \textit{Doc. Math.} 18 (2013), 857--905.

\bibitem [May92]{May92}J. Peter May. Simplicial objects in algebraic topology. Reprint of the 1967 original. \textit{Chicago Lectures in Mathematics}. University of Chicago Press, Chicago, IL, 1992, viii+161 pp.


\bibitem [Nog11]{Nog11} K. Noguchi. The Euler characteristic of acyclic categories. \textit{Kyushu Journal of Math.}, 65(1): 85--99, 2011.

\bibitem [Nog13]{Nog13} K. Noguchi. Euler characteristics of categories and barycentric subdivision. \textit{M\"unster Journal of Math.}, 6(1): 85--116, 2013.

\bibitem [Qui67]{Qui67}Daniel Quillen. Homotopical algebra. \textit{Lecture Notes in Mathematics}, No. 43 Springer-Verlag, Berlin-New York, 1967, iv+156 pp.

\bibitem [Rot64]{Rot64} Gian-Carlo Rota. On the foundations of combinatorial theory. I. Theory of M$\ddot{\mathrm{o}}$bius functions. \textit{Z. Wahrscheinlichkeitstheorie und Verw. Gebiete} 2 1964 340--368 (1964).

\bibitem [Seg68]{Seg68}Graeme Segal. Classifying spaces and spectral sequences. \textit{Inst. Hautes \'Etudes Sci. Publ. Math.} No. 34 1968 105--112. 

\bibitem [Seg74]{Seg74}Graeme Segal. Categories and cohomology theories. \textit{Topology} 13 (1974), 293--312.

\bibitem [Tab05]{Tab05}Goncalo Tabuada. Une structure de cat\'egorie de mod\`eles de Quillen sur la cat\'egorie des dg-cat\'egories. (French) [A Quillen model structure on the category of dg categories] \textit{C. R. Math. Acad. Sci. Paris} 340 (2005), no. 1, 15--19.

\bibitem [Tama]{Tama}Dai Tamaki. Cellular Stratified Spaces I: Face Categories and Classifying Spaces. arXiv:1106.3772v3.

\bibitem [Tamb]{Tamb}Dai Tamaki. Cellular Stratified Spaces II: Basic Constructions. arXiv:1111.4774v3.

\bibitem [Tamc]{Tamc}Mizuki Furuse, Takashi Mukouyama and Dai Tamaki. Totally normal cellular stratified spaces and applications to the configuration space of graphs. arXiv:1312.7368v1.
\end{thebibliography}
\end{document}